\documentclass[12pt]{article}
\usepackage{indentfirst, latexsym, bm, amssymb}

\begin{document}

\newtheorem{theorem}{Theorem}[section]
\newtheorem{corollary}[theorem]{Corollary}
\newtheorem{definition}[theorem]{Definition}
\newtheorem{conjecture}[theorem]{Conjecture}
\newtheorem{question}[theorem]{Question}
\newtheorem{lemma}[theorem]{Lemma}
\newtheorem{proposition}[theorem]{Proposition}
\newtheorem{example}[theorem]{Example}
\newenvironment{proof}{\noindent {\bf
Proof.}}{\rule{3mm}{3mm}\par\medskip}
\newcommand{\remark}{\medskip\par\noindent {\bf Remark.~~}}
\newcommand{\pp}{{\it p.}}
\newcommand{\de}{\em}

\newcommand{\JEC}{{\it Europ. J. Combinatorics},  }
\newcommand{\JCTB}{{\it J. Combin. Theory Ser. B.}, }
\newcommand{\JCT}{{\it J. Combin. Theory}, }
\newcommand{\JGT}{{\it J. Graph Theory}, }
\newcommand{\ComHung}{{\it Combinatorica}, }
\newcommand{\DM}{{\it Discrete Math.}, }
\newcommand{\ARS}{{\it Ars Combin.}, }
\newcommand{\SIAMDM}{{\it SIAM J. Discrete Math.}, }
\newcommand{\SIAMADM}{{\it SIAM J. Algebraic Discrete Methods}, }
\newcommand{\SIAMC}{{\it SIAM J. Comput.}, }
\newcommand{\ConAMS}{{\it Contemp. Math. AMS}, }
\newcommand{\TransAMS}{{\it Trans. Amer. Math. Soc.}, }
\newcommand{\AnDM}{{\it Ann. Discrete Math.}, }
\newcommand{\NBS}{{\it J. Res. Nat. Bur. Standards} {\rm B}, }
\newcommand{\ConNum}{{\it Congr. Numer.}, }
\newcommand{\CJM}{{\it Canad. J. Math.}, }
\newcommand{\JLMS}{{\it J. London Math. Soc.}, }
\newcommand{\PLMS}{{\it Proc. London Math. Soc.}, }
\newcommand{\PAMS}{{\it Proc. Amer. Math. Soc.}, }
\newcommand{\JCMCC}{{\it J. Combin. Math. Combin. Comput.}, }
\newcommand{\GC}{{\it Graphs Combin.}, }

\title{The Tur\'{a}n Number of  Disjoint Copies of Paths
\thanks{
This work is supported by National Natural Science
Foundation of China (No.11271256 and 11531001), The Joint Israel-China Program
(No.11561141001),  Innovation Program of Shanghai Municipal Education Commission (No.14ZZ016) and Specialized Research Fund for the Doctoral Program of Higher Education (No.20130073110075).
\newline \indent $^{\dagger}$Correspondent author:
Xiao-Dong Zhang (Email: xiaodong@sjtu.edu.cn), }}
\author{ Long-Tu Yuan  and Xiao-Dong Zhang$^{\dagger}$   \\
{\small Department of Mathematics, and Ministry of Education }\\
{\small Key Laboratory of Scientific and Engineering Computing, }\\
{\small Shanghai Jiao Tong University} \\
{\small  800 Dongchuan Road, Shanghai, 200240, P.R. China}\\
{\small Email: xiaodong@sjtu.edu.cn, yuanlongtu@sjtu.edu.cn}}
\date{}
\maketitle
\begin{abstract}
The Tur\'{a}n number of a graph $H$, $ex(n,H)$, is the maximum number of edges in a simple graph of order $n$  which does not contain $H$ as a subgraph.  Let $k\cdot P_3$ denote  $k$ disjoint copies of a path on  $3$ vertices. In this paper, we determine the value $ex(n, k\cdot P_3)$ and characterize all extremal graphs. This  extends a result of Bushaw and Kettle  [N. Bushaw and N. Kettle, Tur\'{a}n  Numbers of multiple and equibipartite forests,  Combin. Probab. Comput., 20(2011) 837-853.],  which solved the conjecture proposed by Gorgol in [I. Gorgol. Tur\'{a}n numbers for disjoint copies of graphs. {\it Graphs Combin.},  27 (2011) 661-667.].
\end{abstract}


\vskip 0.5cm

\section{Introduction}
  Our notation in this paper is  standard (see, e.g. \cite{diestel2010}).
Let $G=(V(G), E(G))$ be a simple graph, where $V(G)$ is the vertex set with  $n$ vertices and $E(G)$ is the edge set with size $e(G)$. The {\it degree} of $v\in V(G)$, the number of edges incident to $v$, is denoted by $d_{G}(v)$ and the set of neighbors of $v$ is denoted by $N(v)$. If $u$ and $v$ in $V(G)$ are adjacent, we say that $u$ {\it hits} $v$ or $v$ {\it hits} $u$. If $u$ and $v$ are not adjacent, we say that $u$ {\it misses} $v$ or $v$ {\it misses} $u$.
   If $S\subseteq V(G)$, the induced subgraph of $G$ by $S$ is denoted by $G[S]$.  Let $G$ and $H$ be two disjoint graphs. Denote by $G\bigcup H$  the disjoint union   of $G$ and $H$ and by $k\cdot G$ the disjoint union of  $k$ copies of a graph $G$. Denote by  $G+H$ the graph obtained from $G\bigcup H$ by adding edges between all vertices of $G$ and all vertices of $H$.  Moreover, Denote by $P_{l}$ a path on $l$ vertices and by $M_{t}$ the disjoint union of $\lfloor\frac{t}{2}\rfloor$ disjoint copies of edges and $\lceil \frac{t}{2}\rceil-\lfloor\frac{t}{2}\rfloor$ isolated vertex (maybe no isolated vertex). We often refer to a path by the nature sequence of its vertices, writing, say, $P_l=x_1x_2\ldots x_l$ and calling $P_l$ a path from $x_1$ to $x_l$.

{\it The Tur\'{a}n number} of a graph $H$, $ex(n,H)$, is the maximum number of edges in a graph of order $n$  which does not contain $H$ as a subgraph. Denote by $G_{ex}(n,H)$ a graph on $n$ vertices with $ex(n,H)$ edges  containing no $H$ as a subgraph and call this graph an {\it extremal graph} for $H$. In general, the extremal graph(s) is not unique.

In 1941, Tur\'{a}n proved that the extremal graph without containing $K_r$ as a subgraph is
the Tur\'{a}n graph $T_{r-1}(n)$. Later, Moon \cite{moon1968} (only when $r-1$ divides $n-k+1$) and Simonovits \cite{simonovits1968} showed that  $K_{k-1}+T_{r-1}(n-k+1)$ is the unique extremal graph containing no $k\cdot K_r$  for sufficient large $n$.

In 1959, Erd\H{o}s and Gallai \cite{erdHos1959maximal} proved the following  well known result.
\begin{theorem}\cite{erdHos1959maximal}\label{erdod1959}
If $G$ is a simple graph with $n\ge k$ vertices, then $ex(n,P_{k})\leq \frac{1}{2}(k-2)n$ with equality if and only if $n=(k-1)t$.  Moreover the extremal graph is $\bigcup_{i=1}^{t}K_{k-1}$.
\end{theorem}

Recently, Gorgol \cite{gorgol} studied the Tur\'{a}n number of disjoint copies of any connected graphs. Let $H$ be any connected graph on $l$ vertices.  With aid of the two graphs $G_{ex}(n-kl+1,H)\bigcup K_{kl-1}$ and $G_{ex}(n-k+1,H)+K_{k-1}$, she presented a lower bound for $ex(n,k\cdot H)$. In particular, she proved the following result.
\begin{theorem}\cite{gorgol}\label{gorgol23}
$$\begin{array}{llll}
ex(n, 2\cdot P_3)&=&\lfloor\frac{n-1}{2}\rfloor+n-1, &\mbox{for}\ \ n\ge 9;\\
ex(n, 3\cdot P_3)&=&\lfloor\frac{n}{2}\rfloor+2n-4,&\mbox{for}\ \ n\ge 14.\end{array}$$
\end{theorem}
Further, based on Theorem~\ref{gorgol23} and the lower bound of $k$ disjoint copies of connected graph, she proposed the following conjecture.
\begin{conjecture}\cite{gorgol}\label{con1}
$$ex(n, k\cdot P_3)=\lfloor\frac{n-k+1}{2}\rfloor+(k-1)n-\frac{k(k-1)}{2}$$
for $n$ sufficiently large.
\end{conjecture}
Later, Bushaw and Kettle \cite{bushwa} proved Conjecture~\ref{con1} and characterized all extremal graphs. Their result can be stated as follows.
\begin{theorem}\cite{bushwa}
\label{bushawn7k}
$$ex(n, k\cdot P_3)={k-1  \choose 2}+(n-k+1)(k-1)+\lfloor\frac{ n-k+1}{2}\rfloor\ \mbox{for}\ \ n\ge 7k.$$
Moreover, the extremal graph is $K_{k-1}+M_{n-k+1}.$
\end{theorem}
In fact, Gorgol in \cite{gorgol} also conjecture that the lower bound is sharp for $k\cdot P_3$. Based on the proof of Conjecture \ref{con1}, Bushaw and Kettle  \cite{bushwa} further conjectured that the extremal graph  is unique for $n> 5k-1$. Their conjecture can be stated as follows.
\begin{conjecture}\cite{bushwa, gorgol}\label{con}
$$ex(n, k\cdot P_3)=\left\{\begin{array}{lll}
{3k-1 \choose 2}+\lfloor\frac{n-3k+1}{2}\rfloor, &\mbox{for}\ \ 3k\le n\le 5k-1;\\
{k-1  \choose 2}+(n-k+1)(k-1)+\lfloor\frac{ n-k+1}{2}\rfloor, &\mbox{for}\ \ n\ge 5k-1.
\end{array}\right.$$
\end{conjecture}
In \cite{bushwa}, Bushaw and Kettle also determined the Tur\'{a}n number of $k$ disjoint copies of $P_{l}$ with  $l\ge 4$ and also  characterized all extremal graphs for sufficient large $n$. The related results on the Tur\'{a}n number of paths, forests may be referred to \cite{Balister2008, bielak2014, Faudree1975, lidicksy2013} and the references therein. There are also many hypergraph Tur\'{a}n problems
\cite{Furedi2015, Furedi2014, kostochk2015} of paths and cycles and  some results of the disjoint union of hypergraphs \cite{bushwa2014, gu}.  For Tur\'{a}n numbers of graphs and hypergraphs, there are several excellent surveys \cite{Furedi2013, keevash2011,Mubayi2015} for more information.

   In this paper,  we determine the value $ex(n, k\cdot P_3)$ and characterize all extremal graphs for all $k$ and $n$, which confirms Conjecture~\ref{con}. The main result in this paper can be stated as follows.
\begin{theorem}\label{main}
$$ex(n, k\cdot P_3)=\left\{\begin{array}{lll}
{n \choose 2},& \mbox{for} &  n<3k;\\
 { 3k-1 \choose 2}+\lfloor\frac{n-3k+1}{2}\rfloor, & \mbox{for} & 3k\le n<5k-1;\\
 {3k-1 \choose 2}+k,  & \mbox{for} & n=5k-1;\\
{k-1 \choose 2} +(n-k+1)(k-1)+\lfloor\frac{n-k+1}{2}\rfloor, & \mbox{for} & n>5k-1.
\end{array}\right.$$
  Further, (1). If $n<3k$, then  all extremal graphs are $K_n$, i.e., $G_{ex}(n, k\cdot P_3)=K_n$.

   (2). If $3k\leq n<5k-1$,  then all extremal graphs are $ K_{3k-1}\bigcup M_{n-3k+1}$, i.e.,  $G_{ex}(n, k\cdot P_3)=K_{3k-1}\bigcup M_{n-3k+1}$.

    (3). If $n=5k-1$, then  all extremal graphs are $ K_{3k-1}\bigcup M_{2k}$  and $K_{k-1}+M_{4k},$  i.e., $G_{ex}(n, k\cdot P_3)= K_{3k-1}\bigcup M_{2k}$ or $K_{k-1}+M_{4k}$ ($=M_{4}$ for $k=1$).

  (4). If $n>5k-1$,  then all extremal graphs are  $K_{k-1}+M_{n-k+1}$, i.e., $G_{ex}(n, k\cdot P_3)=K_{k-1}+M_{n-k+1}$.
\end{theorem}

   It should be pointed that the extremal graphs are not unique for $n=5k-1$ and $k\ge 2$, while the extremal graph is unique for otherwise. The rest of this paper is organized as follows. In Section 2, several technique Lemmas are provided. In Section 3, The Proof of Theorem~\ref{main} is presented.
  \section{Several Lemmas}
For the remainder of this section, let $G$ be a simple graph of order $n$ contains $(k-1)\cdot P_3$ and contains no $k\cdot P_3$. In addition, let $H=(k-1)\cdot P_3$ be a $(k-1)$  disjoint paths $x_1y_1z_1,\ldots,x_{k-1}y_{k-1}z_{k-1}$ of $G$ such that $G^{\prime}=G-V(H)$ has the maximum edges, where $G^{\prime}=G-V(H)$ is the subgraph of $G$ by deleting all vertices in $H=(k-1)\cdot P_3$ and incident edges. In order to prove the main result in this paper, we present several technique Lemmas.
\begin{lemma}\label{lemma1}
If $G^{\prime}$  has no edge, i.e., $G^{\prime}$ consists of $t$ isolated vertices  with $t=n-3(k-1)\ge 3$, then
there are at most $p$ edges between $\{w_1, \ldots,  w_p\}\subseteq V(G^{\prime})$ and $\{x_j, y_j, z_j\}$ for $1\le j\le k-1$ and $3\le p\le t$. Moreover, if any two vertices in $G^{\prime}$ both hit at least one vertex in $\{x_j, y_j, z_j\}$, then they both hit $y_j$ and miss $x_j$ and $z_j$, for $1\le j\le k-1$.
\end{lemma}
\begin{proof} Suppose that there are at least $p+1$ edges between $\{w_1, \ldots, w_p\}$ and $\{x_j, y_j, z_j\}$ for some $1\le j\le k-1$ with  $d_F(w_1)\ge \ldots\ge d_F(w_p)$ and $3\le p\le t$, where $F$ is the subgraph of $G$ induced by vertex set $\{x_j, y_j, z_j, w_1, \ldots, w_p\}$.  Then $d_F(w_1)\ge 2$ and $d_F(w_2)\ge 1$. Hence  $w_1$ hits at least two vertices, say $x_j$ and $y_j$ ($y_j$ and $z_j$), or $x_j$ and $z_j$,  in $\{x_j, y_j, z_j\}$.  We consider the following two cases.

{\bf Case 1:} $w_1$ hits  $x_j$ and $y_j$ ($y_j$ and $z_j$). Note $d_G(w_2)\ge 1$. If $w_2$ hits $x_j$, then there exists an $H_1=(k-1)\cdot P_3$ with $(k-1)$ disjoint paths $P_3$: $x_1y_1z_1,$ $ \ldots, x_{j-1}y_{j-1}z_{j-1}, w_1x_jw_2, \ldots,$ $  x_{k-1}y_{k-1} z_{k-1}$ such that
$G^{\prime\prime}=G-V(H_1)$ has at least one edge $y_jz_j$, which contradicts to $G^{\prime}$ having no edges ($G^{\prime}$ has the maximum number of edges). If $w_2$ hits $y_j$ or $z_j$, say $w_2$ hits $y_j$, then there exists an $H_2=(k-1)\cdot P_3$ with $(k-1)$ disjoint paths $P_3$: $x_1y_1z_1, \ldots, x_{j-1}y_{j-1}z_{j-1}, w_2y_jz_j, \ldots, x_{k-1}y_{k-1}z_{k-1}$ such that
$G^{\prime\prime}=G-V(H_2)$ has at least one edge $x_jw_1$, which contradicts to $G^{\prime}$ having no edges. So the assertion holds.

{\bf Case 2:}  $w_1$ hits  $x_j$ and $z_j$.  If $w_2$ hits $x_j$, then there exists an $H_1=(k-1)\cdot P_3$ with $(k-1)$ disjoint paths $P_3$: $x_1y_1z_1, \ldots, x_{j-1}y_{j-1}z_{j-1}, w_1x_jw_2, \ldots, x_{k-1}y_{k-1}z_{k-1}$ such that
$G^{\prime\prime}=G-V(H_1)$ has at least one edge $y_jz_j$, which contradicts to $G^{\prime}$ having no edges. If $w_2$ hits $y_j$ or $z_j$, say $w_2$ hits $y_j$, then there exists an $H_2=(k-1)\cdot P_3$ with $(k-1)$ disjoint paths $P_3$: $x_1y_1z_1, \ldots, x_{j-1}y_{j-1}z_{j-1}, w_2y_jz_j, \ldots, x_{k-1}y_{k-1}z_{k-1}$ such that
$G^{\prime\prime}=G-V(H_2)$ has at least one edge $x_jw_1$, which contradicts to $G^{\prime}$ having no edge.
So the assertion holds.

Moreover, if $w_{i_1}$,$w_{i_2}$ both hit at least one vertex in $\{x_j, y_j, z_j\}$ for $1\le i_1<i_2\le t$, since the subgraph $G[w_{i_1}, w_{i_2}, x_j, y_j, z_j]$ can't contain disjoint union of $P_3$ and an edge, we have that $w_{i_1},w_{i_2}$ both hit $y_j$ and miss $x_j$ and $z_j$. We finish our proof.
\end{proof}

   \begin{lemma}\label{lemma2}
  Suppose that $G^{\prime}$ consists of one edge $u_1v_1$ and $t$ isolated vertices $\{w_1, \ldots,$ $ w_t\}$ with $t=n-3(k-1)-2\ge 2$.

(1). If there are  at most 4 edges between $\{u_1, v_1\}$ and $\{x_j, y_j, z_j\}$ with $1\le j\le k-1$,  then
there are at most 4 edges between $\{u_1, v_1, w_{i_1}, w_{i_2}\}$ and $\{x_j, y_j, z_j\}$ for $ i_1, i_2=1, \cdots, t$ and $i_1\neq i_2$.

(2). If there are  at least 5 edges between $\{u_1, v_1\}$ and $\{x_j, y_j, z_j\}$ with $1\le j\le k-1$,  then
there are at most 6 edges between $\{u_1, v_1\}$ and $\{x_j, y_j, z_j\}$, moreover, there is no edge between $\{w_1, \ldots, w_t\}$ and $\{x_j, y_j, z_j\}$.
\end{lemma}
\begin{proof}
 For  proof of (1),  we consider the following four cases.

{\bf Case 1:} There are at least 3 edges between $\{u_1, v_1\}$ and $\{x_j, y_j, z_j\}$.
 Since $G$ does not contain $k\cdot P_3$, the subgraph $G[x_j, y_j, z_j, u_1, v_1, w_1, w_2]$ of $G$ does not contain $2\cdot P_3$. Hence
 there is no edge between $\{w_{i_1}, w_{i_2}\}$ and $\{x_j, y_j, z_j\}$. So there are at most 4 edges between $\{u_1, v_1\}$ and $\{x_j, y_j, z_j\}$.

 {\bf Case 2:} There are  exactly 2 edges between $\{u_1, v_1\}$ and $\{x_j, y_j, z_j\}$. If $u_1$  hits $x_j$ and $z_j$, or $u_1$ hits  $x_j$ and $v_1$ hits $y_j$, or
 $u_1$ hits $x_j$ and $v_1$ hits $z_j$, then there is no edge between
$\{w_{i_1}, w_{i_2}\}$ and $\{x_j, y_j, z_j\}$, since the subgraph $G[x_j, y_j, z_j, u_1, v_1, w_1, w_2]$ of $G$ does not contain $2\cdot P_3$. If  $u_1$  hits $x_j$ and $y_j$, or $u_1$ and $v_1$ hit $x_j$,  then $w_{i_1}$ and $w_{i_2}$ miss $y_j$ and $z_j$, which implies that there are at most 2 edges between
$\{w_{i_1}, w_{i_2}\}$ and $\{x_j, y_j, z_j\}$. Hence there are at most 4 edges between $\{u_1, v_1\}$ and $\{x_j, y_j, z_j\}$. If $u_1$ and  $v_1$ hit $y_j$, then there are at most one edge between $\{w_{i_1}, w_{i_2}\}$ and $\{x_j, y_j, z_j\}$, since the subgraph $G[x_j, y_j, z_j, u_1, v_1, w_{i_1}, w_{i_2}]$ of $G$ does not contain neither $2\cdot P_3$ nor disjoint union of $P_3$ and two edges (since $G^{\prime}$ has the maximum edges). So there are at most 4 edges between $\{u_1, v_1,w_{i_1},w_{i_2}\}$ and $\{x_j, y_j, z_j\}$.

{\bf Case 3:} There is exactly one edge between $\{u_1, v_1\}$ and $\{x_j, y_j, z_j\}$.
If $u_1$ hits $x_j$ or $z_j$, say $x_j$, then $w_1$ and $w_2$  miss $y_j$ and $z_j$. Hence there are  at most  4 edges between
$\{u_1, v_1, w_{i_1}, w_{i_2}\}$ and $\{x_j, y_j, z_j\}.$ If   $u_1$ hits $y_j$, then  there are at most one edge between $\{w_{i_1}, w_{i_2}\}$ and $\{x_j,  z_j\}$, since the subgraph $G[x_j, y_j, z_j, u_1, v_1, w_{i_1}, w_{i_2}]$ of $G$ does not contain neither  $2\cdot P_3$ nor disjoint union of $P_3$ and two edges. Hence there there are  at most 4 edges between
$\{u_1, v_1, w_{i_1}, w_{i_2}\}$ and $\{x_j, y_j, z_j\}.$ So the assertion holds.

{\bf Case 4:} There is no edge between $\{u_1, v_1\}$ and $\{x_j, y_j, z_j\}$.
 There are at most three edges between $\{w_1,w_2\}$ and $\{x_j,y_j,z_j\}$, since the subgraph $G[x_j, y_j, z_j, w_1, w_2]$ of $G$ does not contain disjoint union of one edge and one $P_3$. So the assertion holds.

Proof of (2). Suppose that there are  at least 5 edges between $\{u_1, v_1\}$ and $\{x_j, y_j, z_j\}$. Then $w_i$ miss $x_j, y_j, z_j$, since  the subgraph $G[x_j, y_j, z_j, u_1, v_1, w_i]$ of $G$ does not contain  $2\cdot P_3$ for $i=1, \ldots, t$. So the assertion holds.
\end{proof}

 \begin{lemma}\label{lemma3}
   Suppose that $G^{\prime}$ consists of $s$ disjoint edges $u_1v_1, \ldots, u_sv_s$ and $t$ isolated vertices $\{w_1, \ldots,$ $ w_t\}$ with $s\ge 2$ and $t=n-3(k-1)-2s$. Moreover, suppose that the number of edges between $\{u_1, v_1\}$ and $\{x_1, y_1, z_1\}$  is the maximum value among the number of edges between $\{u_i, v_i\}$ and $\{x_j, y_j, z_j\}$ for $i=1, \ldots, s$ and $ j=1, \ldots, k-1$.

(1). If there are  at most 4 edges between $\{u_1, v_1\}$ and $\{x_1, y_1, z_1\}$, then
there are at most 4 edges between $\{u_1, v_1, u_i, v_i\}$ and $\{x_j, y_j, z_j\}$ for $i=2, \ldots, s $ and $j=1, \ldots, k-1$.

(2). If there are  at least 5 edges between $\{u_1, v_1\}$ and $\{x_1, y_1, z_1\}$,  then
there are at most 6 edges between $\{u_1, v_1, u_i, v_i\}$ and $\{x_j, y_j, z_j\}$ for $i=2, \ldots, s$ and $j=1, \ldots, k-1$, and there is no edge between $\{u_2, v_2, \ldots, u_s, v_s, w_1, \ldots, w_t\}$ and $\{x_1, y_1, z_1\}$. Moreover,  there are  6 edges between $\{u_1, v_1, u_i, v_i\}$ and $\{x_j, y_j, z_j\}$ if and only if   either there are 6 edges between $\{u_1, v_1\}$ and $\{x_j, y_j, z_j\}$
and there is no edge between  $\{u_i, v_i\}$  and $\{x_j, y_j, z_j\}$, or  there are 6 edges between $\{u_i, v_i\}$ and $\{x_j, y_j, z_j\}$
and there is no edge between  $\{u_1, v_1\}$  and $\{x_j, y_j, z_j\}$. $2\le i\le s, $ $1\le j\le k-1$.
     \end{lemma}

\begin{proof}
 (1).
  If there is at least one vertex, say $u_1$, in $\{u_{1},v_{1}, u_i, v_i\}$ such that $u_1$ hit at least two vertices, say $x_j$ and $y_j$ (or $x_j$ and $z_j$),  in $\{x_j, y_j, z_j\}$, then  there is no edge between $\{ u_i, v_i\}$ and $\{x_{j}, y_{j}, z_{j}\},$  since $G[u_{1},v_{1}, u_i, v_i, x_{j}, y_{j}, z_{j}]$ does not contain $2\cdot P_3$.  Hence there are at most 4 edges  between $\{u_{1},v_{1}, u_i, v_i\}$ and $\{x_{j}, y_{j}, z_{j}\}$.
 If each vertex in $\{u_{1},v_{1}, u_i, v_i\}$  hits at most one vertex in  $\{x_j, y_j, z_j\}$, then  there are at most 4 edges  between $\{u_{1},v_{1}, u_i, v_i\}$ and $\{x_{j}, y_{j}, z_{j}\}$.  So (1) holds.

    (2). If there are at least $5$ edges between $\{u_1, v_1\}$ and $\{x_1, y_1, z_1\}$, then there is no edge between $\{x_1, y_1, z_1\}$ and $\{u_i,v_i,\ldots,u_s,v_s,w_1, \ldots, w_t\}$,  since $G[x_1, y_1, z_1, u_1, v_1, w_p]$ and $G[x_1, y_1, z_1, u_1, v_1, u_i,v_i]$ does not contain $2\cdot P_3$ for $i=2,\ldots,s;p=1, \ldots, t$.
    In addition, it's easy to see that there are at most 6 edges between $\{u_1, v_1, u_i, v_i\}$ and $\{x_j, y_j, z_j\}$.

     If there are 6 edges between $\{u_1, v_1, u_i, v_i\}$ and $\{x_j, y_j, z_j\}$, it is easy to see that either there are 6 edges between $\{u_1, v_1\}$ and $\{x_j, y_j, z_j\}$
and there is no edge between  $\{u_i, v_i\}$  and $\{x_j, y_j, z_j\}$, or  there are 6 edges between $\{u_i, v_i\}$ and $\{x_j, y_j, z_j\}$
and there is no edge between  $\{u_1, v_1\}$  and $\{x_j, y_j, z_j\}$, because  $G[x_j, y_j, z_j, u_1, v_1, u_i,v_i]$ does not contain $2\cdot P_3$. So the assertion holds.
\end{proof}

\section{Proof of Theorem~\ref{main}.}

It is noticed that if $n<3k$, the assertion clearly holds. In addition, if $n=5k-1$, it is easy to see that
$${k-1 \choose 2} +(n-k+1)(k-1)+\lfloor\frac{n-k+1}{2}\rfloor={ 3k-1 \choose 2}+\lfloor\frac{n-3k+1}{2}\rfloor={3k-1 \choose 2}+k.$$

  In order to prove Theorem~\ref{main}, it is sufficient to prove the following statement: if a graph $G$ of order $n$ with ${n \choose 2}$,  ${ 3k-1 \choose 2}+\lfloor\frac{n-3k+1}{2}\rfloor$, ${3k-1 \choose 2}+k$ and  ${k-1 \choose 2} +(n-k+1)(k-1)+\lfloor\frac{n-k+1}{2}\rfloor$  edges for
  $ n<3k$,  $3k\leq n<5k-1$, $n=5k-1$ and  $n>5k-1$ respectively does not contain $k\cdot P_3$, then $G=K_n$,  $K_{3k-1}\bigcup M_{n-3k+1}$,  $ K_{3k-1}\bigcup M_{2k}$  or $K_{k-1}+M_{n-k+1},$  $K_{k-1}+M_{n-k+1}$ for  $ n<3k$,  $3k\leq n<5k-1$, $n=5k-1$ and $n>5k-1$, respectively.

  We will prove the above statement  by double induction on $k$ and $n$. For $k=1$, it is easy to see that the assertion for $n\le 3$.  Moreover a graph of order $4$ which does not contain $P_3$ has to be  $K_2\bigcup M_2=M_4$.
  If a graph of order $n>4$ with $\lfloor\frac{n}{2}\rfloor$ edges does not contain $P_3$, then each component of $G$ is an edge or an isolated vertex. So $G=M_n$ and the assertion  holds for $k=1$ and all $n$. Suppose $k\geq2$ and the assertion holds for less than $k$ and all integers $n$.  Now we will prove the assertion holds for $k$ and all integers $n$.  Clearly,  a graph of order $n<3k$ with ${n \choose 2}$ edges  which does not contain $k\cdot P_3$ has to be $K_n$, i.e., $ex(n, k\cdot P_3)={n \choose 2}$ and $G_{ex}(n, k\cdot P_3)=K_n$ for $n<3k$.  The rest of the proof will be divided into the following three parts.
\subsection{$n=3k.$}
  Let $G$ be a graph of  order $n=3k$ with ${ 3k-1 \choose 2}$ edges which does not contain $k\cdot P_3$.
    Then it is easy to see that $G=K_5\bigcup M_1$  for $k=2$ and $n=3\times 2$.
    So we assume that $k\ge 3$. Since   $e(G)={3k-1 \choose 2}>{3(k-1)-1 \choose 2}+2$ and $n=3k\leq 5(k-1)-1$, $G$ contains $(k-1)\cdot P_3$ as a subgraph
     by  the induction hypothesis for $k-1$ and $n$.  Then there exists an $H=(k-1)\cdot P_3$ such that
     $e(G-V(H))$ is the maximum value, where
      $G-V(H)$ is the  subgraph of  $G$ by deleting all vertices in $H$ and incident edges and $H$ is a $(k-1)$ disjoint $P_3$: $x_1y_1z_1,\ldots,x_{k-1}y_{k-1}z_{k-1}$.  Then $G^{\prime}=G-V(H)$ has three vertices and let $V(G^{\prime})=\{u,v,w\}$.
 Then $e(G^{\prime})\le 1$. Otherwise, $G^{\prime}$ contains a $P_3$  which implies that $G$ contains $k\cdot P_3$, a contradiction.   We consider the following two cases.

{\bf Case 1:} $e(G^{\prime})=1$. Without loss of generality, $u$ hits $v$. Then $w$ misses $u$ and $v$.  By the proof of Lemma~\ref{lemma2}, it is easy to see that
    there are at most 6 edges between $\{u,v,w\}$ and $\{x_j, y_j, z_j\}$ with equality if and only if there is no edge between $w$ and $\{x_j, y_j, z_j\}$, for $1\le j\le k-1$. Hence
      there are at most $6(k-1)$ edges between $\{u,v,w\}$ and $\{x_1, y_1, z_1,$ $\ldots, $ $ x_{k-1}, y_{k-1}, z_{k-1}\}$ with equality if and only if $w$ is an isolated vertex and  there are
 exactly $6(k-1)$ edges  between $\{u,v\}$ and $\{x_1, y_1, z_1,$ $ \ldots,$ $ x_{k-1}, y_{k-1}, z_{k-1}\}$. Therefore
 $${3k-1 \choose 2}=e(G)\le e(G-\{u,v,w\})+6(k-1)+1\le
 {3k-3\choose 2}+6(k-1)+1={3k-1 \choose 2}.$$ So $e(G-\{u,v,w\})={3k-3\choose 2}$ and  there are exactly $6(k-1)$ edges between $\{u,v,w\}$ and $\{x_1, y_1, z_1,$ $\ldots, $ $ x_{k-1}, y_{k-1}, z_{k-1}\}$. Then $w$ is an isolated vertex and  there are
 exactly $6(k-1)$ edges  between $\{u,v\}$ and $\{x_1, y_1, z_1,$ $ \ldots,$ $ x_{k-1}, y_{k-1}, z_{k-1}\}$. Hence $G=K_{3k-1}\bigcup M_{1}$.

{\bf Case 2:} $e(G^{\prime})=0$.  By Lemma~\ref{lemma1}, there are at most 3 edges between $\{u, v, w\}$ and $\{x_{j}, y_{j}, z_{j}\}$ for $1\le j\le k-1$.
Hence there are at most $3(k-1)$ edges $\{u,v, w\}$ and $\{x_1, y_1, z_1,$ $ \ldots,$ $ x_{k-1}, y_{k-1}, z_{k-1}\}$. Then
   $${3k-1 \choose 2}=e(G)\le {3k-3 \choose 2}+3(k-1)<{3k-1 \choose 2}.$$ It is a contradiction. Therefore the assertion holds for $n=3k$. $\Box$

\subsection{$3k<n\leq 5k-1$}

By $e(G)\geq{3k-1 \choose 2}+\lfloor\frac{n-3k+1}{2}\rfloor$ and a simple calculation, it is easy to see
$$ e(G)> \left\{\begin{array}{lll}
{3k-4 \choose 2}+\lfloor\frac{n-3k+4}{2}\rfloor,  &\mbox{for} & 3k< n\leq5(k-1)-1;\\
{k-2 \choose 2}+(n-k+2)(k-2)+\lfloor\frac{n-k+2}{2}\rfloor,
& \mbox{for} & 5(k-1)-1<n\leq 5k-1.\end{array}\right.$$
By the induction hypothesis for $k-1$ and $n$, $G$ contains a $(k-1)\cdot P_{3}$.
 Then there exists an $H=(k-1)\cdot P_3$ such that
     $G^{\prime}=G-V(H)$ has the maximum number of edges, where
     $G-V(H)$ is the subgraph of $G$ by deleting all vertices in $H$ and incident edges and $H$ is a $(k-1)$  disjoint $P_3$: $x_1y_1z_1,\ldots,x_{k-1}y_{k-1}z_{k-1}.$
     Since $G$ does not contain $k\cdot P_3$,  each connected component of $G^{\prime}$
     is an edge or an isolated vertex.
      So $G^{\prime}$ consists of $s$ edges $u_{1}v_{1},u_{2}v_{2},\ldots,u_{s}v_{s}$ and $ t$ isolated vertices $w_{1},w_{2},\ldots,w_{t}$, where $0\le s\le \lfloor\frac{n-3k+3}{2}\rfloor$ and  $ t= n-3k+3-2s$.

 Further, it is not difficulty  to see that   $ s\ge 1$. In fact, if $s=0$, then for any subgraph $H_1$ of $G$ contains $(k-1)\cdot P_3$, $G-V(H_1)$ has no edge. Moreover,  $t=n-3k+3\ge 4$. Since there are at most ${3k-3\choose 2} $ edges in $\{x_1, y_1, z_1, \ldots, x_{k-1}, y_{k-1}, z_{k-1}\}$, there are at least ${3k-1\choose 2}+\lfloor\frac{n-3k+1}{2}\rfloor-{3k-3\choose 2}>6k-5 $  edges between $\{x_1, y_1, z_1, \ldots, x_{k-1}, y_{k-1}, z_{k-1}\}$ and $\{w_1, \ldots, w_t\}$. Hence there exists a $1\le j\le k-1$  such that there are at least 7 edges between $\{x_j, y_j, z_j\}$ and $\{w_1, \ldots, w_t\}$. Without loss of generality, let $w_1$ and $w_2$ hit at least one vertex in $\{x_1, y_1, z_1\}$, respectively. Then by Lemma\ref{lemma1}, $w_1$ and $w_2$ hit $y_1$ and miss $x_1,z_1$. Further, by Lemma~\ref{lemma1},  there are at most 4 edges between $\{x_j, y_j, z_j\}$ and
 $\{w_1, w_2, w_3, w_4\}$ for $j=1, \ldots, k-1$. Let $G_1=G-\{y_1, w_1, w_2, w_3, w_4\}$.  Then
 $$e(G_1)\ge  {3k-1 \choose 2}+\lfloor\frac{n-3k+1}{2}\rfloor-[(n-1)+4(k-2)]
 >{3k-4\choose 2} +    \lfloor\frac{n-3k-1}{2}\rfloor.$$
 On the other hand,  since  $G_1$ is a simple graph of order $n-5$  with $e(G_1)>{3k-4\choose 2}$, we have
 $n-5>3k-4$.  Hence by the induction hypothesis for $k-1$ and $n-5\ge 3(k-1)$, $ G_1$ contains $(k-1)\cdot P_3$. Moreover, the subgraph $G[y_1, w_1, w_2, w_3, w_4]$ contains one $P_3$. So $G$ contains $k\cdot P_3$, which is a contradiction.

      We now consider the following two cases.

{\bf Case 1.}  There exists an edge, says $u_{1}v_{1}$, in $G^{\prime}$ such that there is no edge between $\{u_{1}, v_{1}\}$ and $\{x_1, y_1,z_1, \ldots, x_{k-1}, y_{k-1}, z_{k-1}\}$.

  If $n=3k+1$, let $G_2=G-\{u_{1}\}$ is a subgraph of $G$ with  $3k$ vertices and $e(G_2)\geq{3k-1 \choose 2}$ edges.  Moreover,  $G_2$ does not contain $k\cdot P_3$.
  By the induction hypothesis  on $k$ and $n-1=3k$, we have $G_2=K_{3k-1}\bigcup M_1$ and $v_1$ is an isolated vertex in $G_2$. So $G= K_{3k-1}\bigcup M_{2}$.

  If $n>3k+1$, let $G_3=G-\{u_{1},v_{1}\}$ is a subgraph of $G$ with $n-2\ge 3k$ vertices and  $e(G_3)\geq {3k-1 \choose 2}+\lfloor\frac{n-3k-1}{2}\rfloor$ edges. Moreover, $G_3$ does not contain $k\cdot P_3$.  By the induction hypothesis for $k$ and $n-2$,  we have $e(G_3)={3k-1 \choose 2}+\lfloor\frac{n-3k-1}{2}\rfloor $  and $G_3=K_{3k-1}\bigcup M_{n-3k-1}$.  So $G=K_{3k-1}\bigcup M_{n-3k+1}.$
   Therefore the assertion holds.


{\bf Case 2:}  There is at least one edge between $\{u_{i},v_{i}\}$ and $\{x_1,y_1,z_1, \ldots, x_{k-1}, y_{k-1},$ $ z_{k-1}\}$ for $i=1, \ldots, s$.  Without loss of generality,  we assume that the number of edges between $\{u_1, v_1\}$ and $\{x_1, y_1, z_1\}$ is the maximum value among the numbers of edge between $\{u_i, v_i\}$ and $\{x_j, y_j, z_j\}$  for  $i=1, \ldots, s$ and $j=1, \ldots k-1$.
  Moreover, assume that  $u_1$ hits  a vertex $\alpha $ in $\{x_1, y_1, z_1\}$.

   Further $s\ge 2$. In fact, if $s=1$, then $t=n-3k+3-2s\ge 3k+1-3k+3-2 =2$.
  If there are at most 4 edges between $\{x_1, y_1, z_1\}$ and $\{u_1, v_1\}$,
    then by Lemma~\ref{lemma2}, there are at most  4 edges between
$\{u_1, v_1, w_1, w_2\}$ and $\{x_j, y_j, z_j\},$ for $j=1,2,\ldots,k-1$.
 Let  $G_4=G-\{u_{1},v_{1},\alpha,w_{1},w_{2}\}$.  Then by $n\le 5k-1,$ we have   $$e(G_4)\geq{3k-1 \choose 2}+\lfloor\frac{n-3k+1}{2}\rfloor- 4(k-2)-(n-1)-1>{3k-4\choose 2}+\lfloor\frac{n-3k-1}{2}\rfloor.$$
  On the other hand,  since  $G_4$ is a simple graph of $n-5$  with $e(G_4)>{3k-4\choose 2}$, we have
 $n-5>3k-4$. By the induction hypothesis for $k-1$ and $n-5\ge 3(k-1)$, $G_4$ contains $(k-1)\cdot P_{3}$, which implies that $G$ contains $k\cdot P_{3}$, since $\alpha u_{1}v_{1}$ is a $P_{3}$ in $G$. It is a contradiction.

If there are at least 5 edges between $\{x_1, y_1, z_1\}$ and $\{u_1, v_1\}$, then by Lemma~\ref{lemma2}, there is no edge between $\{x_1, y_1, z_1\}$ and $\{w_1, \ldots, w_t\}$ and  there are at most 6 edges between  $\{x_j, y_j, z_j\}$ and $\{u_1, v_1, w_1, w_2\}$ for $1\le j\le k-1$.
 Let  $G_5=G-\{u_{1},v_{1},\alpha,w_{1},w_{2}\}$.  Then   $$e(G_5)\geq{3k-1 \choose 2}+\lfloor\frac{n-3k+1}{2}\rfloor- 6(k-1)-(3k-4)-1>{3k-4\choose 2}+\lfloor\frac{n-3k-1}{2}\rfloor$$. On the other hand,  since  $G_5$ is a simple graph of $n-5$  with $e(G_5)>{3k-4\choose 2}$, we have
 $n-5>3k-4$. By the induction hypothesis for $k-1$ and $n-5\ge 3(k-1)$, $G_5$ contains $(k-1)\cdot P_{3}$, which implies that $G$ contains $(k-1)\cdot P_{3}$, since   $\alpha u_{1}v_{1}$  is a  $P_{3}$ in $G$. It is  a contradiction.

 By $s\ge 2$, let  $G_6=G-\{u_{1},v_{1},\alpha, u_{2},v_{2}\}$.
   we consider the following two subcases.

   {\bf Subcase 2.1:}
   There are  at most 4 edges between $\{u_1, v_1\}$ and $\{x_1, y_1, z_1\}$.
   Then   $$e(G_6)\geq{3k-1 \choose 2}+\lfloor\frac{n-3k+1}{2}\rfloor-
 [(n-5)+4(k-1)+2]\ge {3k-4\choose 2}+\lfloor\frac{n-3k-1}{2}\rfloor$$ by $n\le 5k-1$.
 On the other hand,  since $G_6$ does not contain $(k-1)\cdot P_{3}$, we have $e(G_6)\le {3k-4\choose 2}+\lfloor\frac{n-3k-1}{2}\rfloor$. Hence $n=5k-1$ and
$e(G_6)= {3k-4\choose 2}+\lfloor\frac{n-3k-1}{2}\rfloor$ with $d_G(\alpha)=n-1$. By the induction hypothesis for $k-1$ and $n-5$, $G_6=K_{3k-4}\bigcup M_{n-3k-1}$  or $G_6=K_{k-2}+M_{n-k-3}$.
If $G_6=K_{3k-4}\bigcup M_{n-3k-1}$, then  $k=2$, otherwise, it is easy to see that  $G$ contains $k\cdot P_3$,  which is a contradiction. So $G=K_1+M_8
$. If $G_6=K_{k-2}+M_{n-k-3}$, then $G=K_{k-1}+M_{n-k+1}$.
Hence  $G=K_{k-1}+M_{n-k+1}$ and $n=5k-1$. So the assertion holds for $k$ and
$n\le 5k-1$.

{\bf Subcase 2.2:} There are at least 5 edges between $\{u_1, v_1\}$ and $\{x_1, y_1, z_1\}$. Then
$$e(G_6)\geq{3k-1 \choose 2}+\lfloor\frac{n-3k+1}{2}\rfloor
 -[6(k-1)+(3k-4)+2]= {3k-4\choose 2}+\lfloor\frac{n-3k-1}{2}\rfloor.$$
   Since $G_6$ does not contain $(k-1)\cdot P_{3}$,
 we have $e(G_6)\le {3k-4\choose 2}+\lfloor\frac{n-3k-1}{2}\rfloor$. Hence $e(G_6)= {3k-4\choose 2}+\lfloor\frac{n-3k-1}{2}\rfloor$ and
  there are 6 edges between $\{u_1, v_1, u_2, v_2\}$ and $\{x_j, y_j, z_j\}$ for $j=1, \ldots, k-1$.  By the condition of Case 2,  there exists a $2\le l\le k-1$ such that
  there is at least one edge between $\{u_2, v_2\}$ and $\{x_l, y_l, z_l\}$.
  Further by Lemma~\ref{lemma3}, there are 6 edges between $\{u_{2},v_{2}\}$ and $\{x_{l},y_{l},z_{l}\}$.
  On the other hand, by the induction hypothesis for $k-1$ and $n-5$, $G_6=K_{3k-4}\bigcup M_{n-3k-1}$  or $G_6=K_{k-2}+M_{n-k-3}$ with $n=5k-1$. Since $x_l y_l z_l$ is one of the $k-2$ disjoint $P_3$ in $G_6$ and both of $u_2,v_2$ hit all vertices of $x_l y_l z_l$, it is easy to see that
  $G$ contains $k\cdot P_3$, which is a contradiction.
 Therefore the assertion holds for $k$ and $n\le 5k-1$.

\subsection{$n>5k-1$}

Since $n>5(k-1)-1$ and $e(G)\geq {k-1 \choose 2} +(n-k+1)(k-1)+\lfloor\frac{n-k+1}{2}\rfloor>    {k-2 \choose 2} +(n-k+2)(k-2)+\lfloor\frac{n-k+2}{2}\rfloor$
 $G$ contains $(k-1)\cdot P_{3}$ by the induction hypothesis for $k-1$ and $n$.
     Then there exists an $H=(k-1)\cdot P_3$ such that
     $G^{\prime}=G-V(H)$ has the maximum number of edges, where
     $G-V(H)$ is the subgraph of $G$ by deleting all vertices in $H$ and incident edges and $H$ is a $(k-1)$  disjoint $P_3$: $x_1y_1z_1,\ldots,x_{k-1}y_{k-1}z_{k-1}.$
          Moreover each connected component of $G^{\prime}$
     is an edge or an isolated vertex,  Since $G$ does not contain $k\cdot P_{3}$.
           So we assume that $G^{\prime}$ consists of $s$ edges $u_{1}v_{1},u_{2}v_{2},\ldots,u_{s}v_{s}$ and $ t$ isolated vertices $w_{1},w_{2},\ldots,w_{t}$, where $0\le s\le \lfloor\frac{n-3k+3}{2}\rfloor$ and  $ t= n-3k+3-2s$.

           It is not difficult to see that  $ s\ge 1$. Suppose that $s=0 $.
                       Since there are at most ${3k-3\choose 2} $ edges in $\{x_1, y_1, z_1, \ldots, x_{k-1}, y_{k-1}, z_{k-1}\}$, there are at least ${k-1 \choose 2} +(n-k+1)(k-1)+\lfloor\frac{n-k+1}{2}\rfloor-{3k-3\choose 2}>6k-5$  edges between $\{x_1, y_1, z_1, \ldots, x_{k-1}, y_{k-1}, z_{k-1}\}$ and $\{w_1, \ldots, w_t\}$. Hence there exists a $1\le j\le k-1$  such that there are at least 7 edges between $\{x_j, y_j, z_j\}$ and $\{w_1, \ldots, w_t\}$. Without loss of generality, let $w_1$ and $w_2$ hit at least one vertex in $\{x_1, y_1, z_1\}$, respectively. Then $w_1$ and $w_2$ hit $y_1$, $\{w_1,w_2,w_3,w_4\}$ misses $\{x_1,z_1\}$ by the maximum value $e(G^{\prime})$. By Lemma~\ref{lemma1},  there are at most 4 edges between $\{w_1,w_2,w_3,w_4\}$ and $\{x_j, y_j, z_j\}$ for $j=1,2,\ldots,k-1$. Let $G_7=G-\{y_1, w_1, w_2, w_3, w_4\}$. 
 \begin{eqnarray*}
 e(G_7)&=&  {k-1 \choose 2}+(n-k+1)(k-1)+\lfloor\frac{n-k+1}{2}\rfloor-[(n-1)+4(k-2)]\\
& >&{k-2\choose 2} +(n-k-3)(k-2)+    \lfloor\frac{n-k-3}{2}\rfloor.\end{eqnarray*}
   Hence by the induction hypothesis for $k-1$ and $n-5\ge 3(k-1)$, $ G_7$ contains $(k-1)\cdot P_3$. Hence $G$ contains $k\cdot P_3$, which is a contradiction. So $s\ge 1$.

  Further, we can show that there is at least one edge between $\{u_i, v_i\}$ and $\{x_1, y_1, z_1,$ $ \ldots, x_{k-1}, y_{k-1}, z_{k-1}\}$ for $i=1, \cdots, s$.
In fact, suppose that there is no edge between $\{u_i, v_i\}$ and $\{x_1, y_1, z_1, \ldots, x_{k-1}, y_{k-1}, z_{k-1}\}$ for some $1\le i\le s$.
     If $n\geq 5k+1$, let $G_8=G-\{u_{i},v_{i}\}$. Then \begin{eqnarray*}
     e(G_8)&\geq & {k-1 \choose 2} +(n-k+1)(k-1)+\lfloor\frac{n-k+1}{2}\rfloor-1\\
     &>&{k-1 \choose 2} +(n-k-1)(k-1)+\lfloor\frac{n-k-1}{2}\rfloor.
     \end{eqnarray*}
      By the induction hypothesis for $k$ and $n-2$, $G_8$ contains $k\cdot P_{3}$, which is a contradiction.   If $n=5k$ and $k\geq3$, let $G_9=G-\{u_{i}\}$. Then
      $$e(G_9)\geq {k-1 \choose 2} +(n-k+1)(k-1)+\lfloor\frac{n-k+1}{2}\rfloor-1>{k-1 \choose 2} +(n-k)(k-1)+\lfloor\frac{n-k}{2}\rfloor.$$ By the induction hypothesis for $k$ and $n-1$, $G_9$ contains $k\cdot P_{3}$. If $n=10, k=2$ and $G_9=G-\{u_{i}\}$,  then $G_9$ does not contain $2\cdot P_3$.  By  the induction hypothesis for $k=2$ and $n=9$, $G=K_1+M_8$ or $K_5\bigcup M_4$, which implies $G$ contains $2\cdot P_3$. It is a contradiction.

 Without loss of generality, the number of edges between $\{u_1, v_1\}$ and $\{x_1, y_1, z_1\}$  is the maximum value among the number of edge between $\{u_i, v_i\}$ and $\{x_j, y_j, z_j\}$ for $i=1, \cdots, s$ and $ j=1, \ldots, k-1$.

 It is not difficult to see that $s\ge 2$. In fact, Suppose that $s=1$.  Since there is at least one edge between $\{u_1,v_1\}$ and $\{x_1, y_1, z_1\}$, we assume that $u_1$ hits one vertex $\alpha$ in $\{x_1, y_1, z_1\}$.
 If there are at most 4 edges between $\{x_1, y_1, z_1\}$ and $\{u_1, v_1\}$,  then  by Lemma~\ref{lemma2}, there are at most  4 edges between
$\{u_1, v_1, w_1, w_2\}$ and $\{x_j, y_j, z_j\}.$  Let  $G_{10}=G-\{u_{1},v_{1},\alpha,w_{1},w_{2}\}$.  Then
 \begin{eqnarray*}
e(G_{10})&\geq & {k-1 \choose 2}+(n-k+1)(k-1)+\lfloor\frac{n-k+1}{2}\rfloor- 4(k-2)-(n-1)-1\\
&>& {k-2\choose 2}+(n-k-3)(k-2)+\lfloor\frac{n-k-3}{2}\rfloor.
\end{eqnarray*}
  By the induction hypothesis for $k-1$ and $n-5\ge 5(k-1)-1$, $G_{10}$ contains $(k-1)\cdot P_{3}$, which implies that $G$ contains $k\cdot P_{3}$, since   $\alpha u_{1}v_{1}$   is a  $P_{3}$ in $G$. It is  a contradiction.

If there are at least 5 edges between $\{u_1, v_1\}$ and $\{x_1, y_1, z_1\}$, then by Lemma~\ref{lemma2},  there are at most 6 edges  between $\{u_{1},v_{1}, w_1, w_2\}$ and $\{x_{j}, y_{j}, z_{j}\}$ for $j=1, \ldots, k-1$ and  there is no edge between $\{w_1, \ldots, w_t\}$ and $\{x_1, y_1, z_1\}$.
 Let  $G_{11}=G-\{u_{1},v_{1},\alpha,w_{1},w_{2}\}$. Then
  \begin{eqnarray*}
e(G_{11})&\geq & {k-1 \choose 2}+(n-k+1)(k-1)+\lfloor\frac{n-k+1}{2}\rfloor- 6(k-1)-(3k-4)-1\\
&>& {k-2\choose 2}+(n-k-3)(k-2)+\lfloor\frac{n-k-3}{2}\rfloor.
\end{eqnarray*}
 By the induction hypothesis for $k-1$ and $n-5\ge 3(k-1)$, $G_{11}$ contains $(k-1)\cdot P_{3}$, which implies that $G$ contains $k\cdot P_{3}$, since   $\alpha u_{1}v_{1}$  is a  $P_{3}$ in $G$. It is  a contradiction.

    We consider the following two cases.

    {\bf Case 1.} There are  at most 4 edges between $\{u_1, v_1\}$ and $\{x_1, y_1, z_1\}$. Then  by Lemma~\ref{lemma3},  there are at most 4 edges  between $\{u_{1},v_{1}, u_2, v_2\}$ and $\{x_{j}, y_{j}, z_{j}\}$, for $1\le j\le k-1$.
Let  $G_{12}=G-\{u_{1},v_{1},\alpha, u_{2},v_{2}\}$, where $\alpha$ is a vertex in $\{x_1, y_1, z_1\}$ which hits $u_1$ or $v_1$.  There are at most $n-5$ edges between $\alpha$ and $V(G)-\{u_{1},v_{1},\alpha, u_{2},v_{2}\}$.
    Then
    \begin{eqnarray*}
    e(G_{12})&\geq& {k-1 \choose 2}+(n-k+1)(k-1)+\lfloor\frac{n-k+1}{2}\rfloor-
 [(n-5)+4(k-1)+2]\\
 &=&  {k-2\choose 2}+[n-5-(k-1)+1](k-2)+\lfloor\frac{n-5-3(k-1)+1}{2}\rfloor.
  \end{eqnarray*}
Note that $G_{12}$ does not contain $(k-1)\cdot P_3$. By the induction hypothesis on $k-1$ and $n-5>5(k-1)-1$, $G_{12}=K_{k-2}+M_{n-k-3}$.
    Moreover, $\alpha$ hits all vertices in $G_{12}$ and there are $4(k-1)$ edges between $G_{12}+\alpha$ and $\{u_1,v_1,u_2,v_2\}$. hence $G=K_{k-1}+M_{n-k+1}$.

    {\bf Case 2.} There are at least 5 edges between $\{u_1, v_1\}$ and $\{x_1, y_1, z_1\}$. By Lemma~\ref{lemma3}, there are at most 6 edges  between $\{u_{1},v_{1}, u_2, v_2\}$ and $\{x_{j}, y_{j}, z_{j}\}$ for $j=1, \ldots, k-1$ and  there is no edge between $\{u_3,v_3, \ldots, u_{s},v_{s}, w_1, \ldots, w_t\}$ and $\{x_1, y_1, z_1\}$.   Let  $G_{13}=G-\{u_{1},v_{1},\alpha, u_{2},v_{2}\}$, where $\alpha$ is a vertex in $\{x_1, y_1, z_1\}$ which hits $u_1$ or $v_1$.  Note that there is no edge between $\alpha$ and $\{u_3,v_3, \ldots, u_{s},v_{s},w_1, \ldots, w_t\}$, there are at most $3k-4$ edges between $\alpha$ and $\{x_1, y_1, z_1, \ldots, x_{k-1}, y_{k-1}, z_{k-1}\}\setminus\{\alpha\}$. Hence
    \begin{eqnarray*}
    e(G_{13})&\geq& {k-1 \choose 2}+(n-k+1)(k-1)+\lfloor\frac{n-k+1}{2}\rfloor-
 [6(k-1)+(3k-4)+2]\\
 &> &  {k-2\choose 2}+[n-5-(k-1)+1](k-2)+\lfloor\frac{n-5-(k-1)+1}{2}\rfloor
  \end{eqnarray*}
by $n>5k-1$. By the induction hypothesis on $k-1$ and $n-5>5(k-1)-1$, $G_{13}$ contains $(k-1)\cdot P_3$, which implies $G$ contains $k\cdot P_3$. It is a contradiction.  Hence the assertion holds for $k$ and all $n$. By the induction principle, the assertion holds for all positive integers $k$ and $n$.


\begin{thebibliography}{1}

\bibitem{Balister2008}P.N.~Balister,E.~Gy\H{o}ri,J.~Lehel and R.~H.Schelp, Connected graphs without long paths. {\it  Discrete Math.}, {\bf 308}(2008), 4487-4494.

\bibitem{bielak2014} H.~Bielak and S.~Kieliszek,  The Tur\'{a}n number of the graph $3P_4$, {\it  Ann. Univ. Mariae Curie-Sk{\l}odowska Sect. A},  {\bf 68(1)}(2014), 21-29.

\bibitem{bushwa} N.~Bushaw and N.~Kettle, Tur\'{a}n numbers of multiple paths and equibipartite forests, {\it  Combin. Probab. Comput. },  {\bf 20}(2011), 837-853.

\bibitem{bushwa2014} N.~Bushaw and N.~Kettle, Tur\'{a}n numbers for forests of paths in hypergraphs, {\it  SIAM J. Discrete Math.},  {\bf 28(2)}(2014), 711-721.

\bibitem{diestel2010}R.~Diestel, {\it Graph Theory, third edition}, Graduate Texts in Mathematics, 173. Springer-Verlag, Berlin, 2005.

\bibitem{erdHos1959maximal} P.~Erd\H{o}s and T.~Gallai,  On maximal paths and circuits of graphs, {\it  Acta Math. Acad. Sci. Hungar,}  {\bf 10(3)}(1959),  337-356.



\bibitem{Faudree1975}R.~J.~Faudree and R.~H.~Schelp, Path Ramsey numbers in multicolourings. {\it  J. Combin. Theory  Ser. B,} {\bf 19}(1975), 150-160.

\bibitem{Furedi2015} Z.~F\"{u}redi, T.~Jiang and R.~Seiver, Exact solution of the hypergraph Turan problem for k-uniform linear paths, {\it Combinatorica,} {\bf 34(3)}(2014), 299-322.

\bibitem{Furedi2014} Z.~F\"{u}redi and T.~Jiang, Hypergraph Turan numbers of linear cycles, {\it J. Combin. Theory Ser. A}, {\bf 123}(2014),  252-270.

\bibitem{Furedi2013} Z.~F\"{u}redi and M.~Simonovits, The history of degenerate (bipartite) extremal graph problems, {\it Erd\H{o}s centennial, 169-264, Bolyai Soc. Math. Stud., 25,}  J\'{a}nos Bolyai Math. Soc., Budapest, 2013.

\bibitem{gorgol} I.~Gorgol,  Turan numbers for disjoint copies of graphs, {\it Graphs Combin.},  {\bf 27 }(2011),  661-667.

\bibitem{gu} R. Gu, X. Li and Y. Shi, Turan numbers of vertex disjoint cycles. arXiv:1305.5372, 2013.
    
\bibitem{keevash2011} P.~Keevash, Hypergraph Turan problems, in {\it Surveys in Combinatircs 2011,} London Math. Soc. Lecture Note Ser. 392, Cambridge University Press, Cambrigde, 2011, pp. 83-139.

\bibitem{kostochk2015} A.~Kostochka, D.~Mubayi and J.~Verstra\"{e}te,  Tur\'{a}n problems and shadows I: paths and cycles, {\it  J. Combin. Theory Ser. A, } {\bf 129}(2015),  57-79.

\bibitem{lidicksy2013}B.~Lidick\'{y}, H.~Liu and  C.~Palmer, On the Tur\'{a}n number of forests, {\it Electron. J. Combin.,}  {\bf 20(2)}(2013),  Paper 62, 13pp.
    
\bibitem{moon1968} J.~W.~Moon,  On independent complete subgraphs in a graph, {\it  Canad. J. Math. }, {\bf 20}(1968), 95-102.

\bibitem{Mubayi2015} D.~Mubayi and J.~Verstra\"{e}te,  A survey of Tur\'{a}n problems for expansions. arXiv e-prints, 2015.

\bibitem{simonovits1968} M.~Simonovits, A method for solving extremal problems in extremal graph theory, {\it In Theory of Graphs,} P.~Erd\H{o}s and G. Katona eds, Academic Press, (1968), 279-319.

\end{thebibliography}
\end{document}